\documentclass[11pt,a4paper,reqno]{amsart}
\usepackage[T1]{fontenc}
\usepackage[utf8]{inputenc}
\usepackage{lmodern}
\usepackage{amsmath,amssymb,amsthm,mathtools}
\usepackage[a4paper,textwidth=154mm,textheight=235mm,centering]{geometry}
\usepackage{microtype}
\usepackage{xcolor}
\definecolor{linkblue}{RGB}{30,64,105}
\usepackage[colorlinks=true,linkcolor=linkblue,citecolor=linkblue,urlcolor=linkblue]{hyperref}
\numberwithin{equation}{section}
\newtheorem{theorem}{Theorem}[section]
\newtheorem{lemma}[theorem]{Lemma}

\theoremstyle{remark}
\newcommand{\R}{\mathbb R}
\newcommand{\qf}{\mathfrak q}
\newcommand{\hf}{\mathfrak h}
\newcommand{\ip}[2]{\langle #1,#2\rangle}
\DeclareMathOperator{\Span}{span}
\DeclareMathOperator{\Rea}{Re}
\allowdisplaybreaks[1]
\newcommand{\beq}{\begin{equation} \begin{split}}
		\newcommand{\eeq}{\end{split} \end{equation}}

\newcommand\Omg{\Omega}
\newcommand\ext{{\rm ext}}

\makeatletter
\def\section{\@startsection{section}{1}\z@{.9\linespacing\@plus\linespacing}%
	{.7\linespacing} {\fontsize{13}{14}\selectfont\bfseries\centering}}
\def\paragraph{\@startsection{paragraph}{4}%
	\z@{0.3em}{-.5em}%
	{$\bullet$ \ \normalfont\itshape}}

\@addtoreset{equation}{section}
\makeatother

\renewcommand\and{\qquad\text{and}\qquad}

\newcommand{\comm}[1]{}

\def\bm1{\mathbbm{1}}

\def\C{\cC}

\def\p{\partial}

\def\arr{\rightarrow}

\def\tt{\theta}
\def\aa{\alpha}
\def\lm{\lambda}

\def\ii{{\mathsf{i}}}
\def\p{\partial}

\def\kp{\kappa}

\def\dd{{\,\mathrm{d}}}

\newcounter{counter_a}

\newcommand{\cf}{{\it cf.}\,}

\numberwithin{figure}{section}
\numberwithin{equation}{section}
\theoremstyle{plain}
\newtheorem*{thm*}{Theorem}
\newtheorem{thm}{Theorem}[section]

\theoremstyle{remark}
\newtheorem{remark}[thm]{Remark}
\theoremstyle{plain}

\newcommand{\beu}{\begin{equation*}}
	\newcommand{\eeu}{\end{equation*}}
\newcommand{\besu}{\begin{equation*}
		\begin{aligned}}
		\newcommand{\eesu}{\end{aligned}
\end{equation*}}
\newcommand{\bes}{\begin{equation}
		\begin{aligned}}
		\newcommand{\ees}{\end{aligned}
\end{equation}}

\newcommand\cB{\mathcal B}

\newcommand\frh{\mathfrak h}

\newcommand\eps{\varepsilon}

\newcommand\wh{\widehat}

\newcommand\void[1]{}

\def\eps{\varepsilon}

      \def\dC{{\mathbb C}}

      \def\dR{{\mathbb R}}
\def\dS{{\mathbb S}}

   \def\cB{{\mathcal B}}

\newcommand{\dom}{\mathrm{dom}\,}

\renewcommand{\C}{\mathbb C}

\usepackage[normalem]{ulem}
\definecolor{DarkGreen}{rgb}{0,0.5,0.1} 
\definecolor{DarkBlue}{rgb}{0,0.1,0.7} 

\newcommand\soutD{\bgroup\markoverwith
{\textcolor{DarkGreen}{\rule[.5ex]{2pt}{1pt}}}\ULon}
\newcommand\soutP{\bgroup\markoverwith
{\textcolor{blue}{\rule[.5ex]{2pt}{1pt}}}\ULon}
\newcommand{\Hm}[1]{\leavevmode{\marginpar{\tiny%
$\hbox to 0mm{\hspace*{-0.5mm}$\leftarrow$\hss}%
\vcenter{\vrule depth 0.1mm height 0.1mm width \the\marginparwidth}%
\hbox to
0mm{\hss$\rightarrow$\hspace*{-0.5mm}}$\\\relax\raggedright #1}}}

\title[Optimization of the second Robin eigenvalue on a planar exterior domain]{Optimization of the second Robin eigenvalue on a planar exterior domain under fixed area constraint}
\author[D.\ Krej\v{c}i\v{r}\'{i}k]{David Krej\v{c}i\v{r}\'{i}k}
\address{(D.\ Krej\v{c}i\v{r}\'{i}k) Department of Mathematics, Faculty of Nuclear Sciences and Physical Engineering, Czech Technical University in Prague, Trojanova 13, 120 00, Prague, Czech Republic}
\email{david.krejcirik@fjfi.cvut.cz}

\author[V.\ Lotoreichik]{Vladimir Lotoreichik}
\address{(V.\ Lotoreichik) Department of Mathematics, Faculty of Nuclear Sciences and Physical Engineering, Czech Technical University in Prague, Trojanova 13, 120 00, Prague, Czech Republic}
\email{vladimir.lotoreichik@fjfi.cvut.cz}

\date{25 October 2026}
\begin{document}
\begin{abstract}
We consider the Robin Laplacian on a planar
exterior domain with a negative boundary parameter. For a negative boundary parameter of a sufficiently large absolute value this operator
has at least two negative eigenvalues below the bottom of the essential spectrum.
We prove that under a fixed area constraint and convexity assumption on the bounded complement
of the exterior domain the second Robin eigenvalue is maximized by the exterior of the disk.	
The inequality is strict for every exterior of a non-circular domain whenever
the comparison disk has a negative second eigenvalue. The proof modifies the
parallel-coordinate method by allowing the angular phase of a trial function
to vary with the distance from the boundary.  This result resolves the first item in~\cite[Conjecture 1.2]{KL24} by the authors  under an additional convexity assumption. 
Relaxing the convexity and the perimeter constraint 
remain open problems.
\end{abstract}
\maketitle

\section{Introduction and main result}
Optimization of eigenvalues with respect 
to the shape of the domain is one of the central topics in spectral geometry. 
The most classical result in this area is 
the hundred-year-old Faber--Krahn inequality stating 
that the ball minimizes the first Dirichlet eigenvalue 
under fixed volume constraint.
Changing the boundary conditions makes the optimization 
problem much younger and still unresolved in several aspects.

The present paper is concerned with the Robin problem
\begin{equation}\label{problem}
\begin{cases}
-\Delta u=\lambda u,&\text{in}\,\,\Omega, \\
\partial_\nu u+\alpha u= 0,&\text{on}\,\, \partial\Omega,
\end{cases}
\end{equation}
where $\Omg\subset\dR^d$ with $d\ge2$ 
is a Lipschitz domain, $\nu$ is its outward unit normal
and~$\alpha$ is a real parameter. 
For positive~$\alpha$, 
the fact that the ball is the minimizer of the lowest eigenvalue 
among all bounded domains of fixed volume 
was proved by Bossel~\cite{Bossel} in 1986 in the planar case
and extended by Daners~\cite{Daners}  in 2006 to higher dimensions. 

The case of negative~$\alpha$ remains mysterious.
For $\alpha<0$, 
Bareket~\cite{Bareket} conjectured in 1977 
that the disk instead maximizes the lowest Robin eigenvalue under fixed area constraint for planar domains. 
However, it turned out in 2015 \cite{FK}
that annuli, and higher-dimensional spherical shells, outperform the equal-volume ball for sufficiently negative~$\alpha$. The corrected conjecture~\cite{AFK} involves simply-connectedness assumption in two dimensions. Freitas and the first author proved in~\cite{FK} the disk maximality for smooth planar domains when $\alpha$ is sufficiently close to zero, uniformly at fixed area; local maximality results also support the conjecture \cite{FNT}. Optimality of the ball for the lowest eigenvalue in the weak coupling regime and under fixed volume constraint was extended in~\cite{GKP} to convex domains in every dimension, with the threshold depending only on volume and dimension. Global validity for every $\alpha<0$ remains open for simply connected planar domains and for convex domains in higher dimensions. Under a different constraint, fixed perimeter, disk maximality for planar domains and ball maximality among convex sets hold for all $\alpha<0$ \cite{AFK,BFNT}. The optimization results under fixed perimeter constraint are extended in~\cite{AGM24} to the Robin $p$-Laplacian.

For the second Robin eigenvalue 
$\lambda_2^\aa(\Omega)$ of~\eqref{problem},
Freitas and Laugesen~\cite{FLW} in 2021 
established the following bound 
\[
\lambda_2^\aa(\Omega)\leq\lambda_2^\aa(\cB_R),
\qquad -\frac{d+1}{dR}\leq\alpha<0,
\]
where~$\Omega$ is bounded
and~$\cB_R$ denotes the ball of radius $R>0$
such that $|\cB_R|=|\Omega|$;
moreover, the equality holds only for the balls.
While it is known that the
ball maximality fails for sufficiently negative $\alpha$, the full transition and subsequent optimizers remain unresolved. Optimization results for the second Robin eigenvalue with negative boundary parameter in a more restricted interval of boundary parameters with boundary parameter scaled by the perimeter in the case of planar domains are obtained in~\cite{FLS} using the method of conformal maps. These considerations are extended to the third Robin eigenvalue for the negative boundary parameter in a certain area-dependent interval in~\cite{GL}. The latter result on the third Robin eigenvalue is also extended to an area-dependent interval of positive boundary parameters in \cite{KL}. The optimization of the second Robin eigenvalue with negative boundary parameter on domains in constant curvature manifolds was addressed in~\cite{LL25}.

The negative boundary parameter makes the optimization problem for the Robin eigenvalues meaningful also on exterior domains. Let us discuss this setting in more detail.
Let $\Omega\subset\R^2$ be a bounded simply-connected 
smooth domain. The exterior domain
$\Omega^{\ext}:=\R^2\setminus\overline\Omega$ is the complement of $\Omg$. It is an unbounded domain with compact boundary. 
Negative Robin boundary
conditions give rise to the existence of finitely many negative eigenvalues for the Robin Laplacian on $\Omega^{\ext}$, while 
the essential spectrum of this operator coincides with $[0,\infty)$. The authors of the present paper
established that the disk is a maximizer for the first Robin eigenvalue on exterior 
planar domains
\cite{KL18,KL20} under both fixed area and fixed perimeter constraints.
The paper~\cite{KL18} dealt with convex domains, while the convexity assumption was removed later in~\cite{KL20}. These optimization results for the lowest Robin eigenvalue on planar exterior domains are generalized in~\cite{BGS26} to the Robin $p$-Laplacian. 

We subsequently conjectured the corresponding
fixed-area and fixed-perimeter 
optimality of the disk
for the second eigenvalue, too~\cite[Conjecture~1.2]{KL24}. In the same paper, we proved an isoperimetric inequality for the second eigenvalue under a fixed boundary elastic energy constraint and under additional assumption on the convexity of $\Omg$, where the proof was based on an improvement of the method developed in~\cite{EL22}. We remark that in space dimensions $d\ge3$ the exterior of the ball is not a global maximizer of the lowest Robin eigenvalue under fixed volume or perimeter constraints~\cite{KL18}. It becomes a global maximizer in the class of exteriors of bounded convex domains under a constraint involving the mean curvature~\cite{KL20}. However, locally the ball is still a maximizer for the lowest eigenvalue under fixed volume constraint in all space dimensions~\cite{B25}.
Exterior domains can be viewed as an extreme case of doubly-connected domains, for which the optimization of the Robin eigenvalues was analyzed in~\cite{ABD25}; see also the references therein.  

The present paper
positively resolves the fixed-area conjecture of \cite{KL24},
at least for convex domains. We employ
the method of~\cite{KL24} based on the construction of a two-dimensional trial space by means of parallel coordinates on $\Omg^\ext$. This
trial space is a span of two trial functions.
The first trial function remains independent of the boundary coordinate as in~\cite{KL24}, but
now uses the radial profile of the second eigenfunction for the exterior of the disk instead of the first eigenfunction for the exterior of the disk. The second trial function employs the same radial profile as in~\cite{KL24}, but uses a
distance-dependent angular phase instead of distance-independent angular phase used in~\cite{KL24}. This idea of distance-dependent angular phase was suggested by ChatGPT 6 Astra.

To precisely formulate the main result of this paper,
let us consider
the closed, densely defined, symmetric and semi-bounded sesquilinear form
\begin{equation}\label{eq:form}
 \qf_\alpha^{\Omg^\ext}[u,v]
 =\int_{\Omega^{\ext}}\nabla u\cdot\overline{\nabla v}\dd x
   +\alpha\int_{\partial\Omega}u\overline v\dd s,
 \qquad \dom\qf_\alpha^{\Omg^\ext} =  H^1(\Omega^{\ext}),
\end{equation}
{which is naturally associated with the problem~\eqref{problem}
in the Hilbert space $L^2(\Omg^\ext)$. 
Writing 
$\qf_\alpha^{\Omg^\ext}[u]=\qf_\alpha^{\Omg^\ext}[u,u]$, 
the associated second min--max level is 
\begin{equation}\label{eq:minmax}
 \lambda_2^\alpha(\Omega^{\ext})
 =\inf_{\substack{\mathcal L\subset H^1(\Omega^{\ext})\\
                         \dim\mathcal L=2}}
   \ \sup_{0\ne u\in\mathcal L}
             \frac{\qf_\alpha^{\Omg^\ext}[u]}{\|u\|_{L^2(\Omega^{\ext})}^2}.
\end{equation}
The self-adjoint operator in $L^2(\Omg^\ext)$ associated with $\qf_\alpha^{\Omg^\ext}$ is the Robin Laplacian $-\Delta_\aa^{\Omg^\ext}$ on $\Omg^\ext$.
The essential spectrum of $-\Delta_\aa^{\Omg^\ext}$ is given by
$[0,\infty)$. The second min-max level \eqref{eq:minmax} equals zero when there are fewer than
two negative eigenvalues; see \cite[Section~2]{KL18} and
\cite[Section~2.1]{KL24}. In particular,
$\lambda_2^\alpha(\Omega^{\ext})\le0$ for every $\alpha\in\R$. We remark that there exists a critical constant $\alpha_\star(\Omega^\ext) <0$ such that $\lm_2^\alpha(\Omg^\ext) < 0$ if, and only if, $\aa < \aa_\star(\Omg^\ext)$; \cf~\cite[Section~2.1]{KL24}.

Our main result is 
the following fixed-area
inequality for the second eigenvalue of the Robin Laplacian on the planar exterior domain $\Omg^\ext$.
\begin{theorem}\label{thm:main}
Let $\Omega\subset\R^2$ be a bounded convex domain with $C^\infty$-smooth boundary, and
let $\cB_R$ be a disk of radius $R$ satisfying $|\Omega|=|\cB_R|$. Then, for every
$\alpha\in\R$,
\begin{equation}\label{eq:main}
 \lambda_2^\alpha(\Omega^{\ext})
 \le \lambda_2^\alpha(\cB_R^{\ext}).
\end{equation}
If $\alpha<-1/R$, equality in \eqref{eq:main} holds if and only if $\Omega$
is a disk, up to translation.
\end{theorem}

For $\alpha\ge-1/R$, the right-hand side of \eqref{eq:main} is zero.
The substantive part of the theorem is therefore the range $\alpha<-1/R$.
The proof will construct a two-dimensional space on which
$\qf_\alpha^{\Omg^\ext}-\lambda_2^\alpha(\cB_R^{\ext})\|\cdot\|^2_{L^2(\Omg^\ext)}$ is strictly negative when $\Omg$ is not a disk. We remark that an isoperimetric
inequality $\aa_\star(\Omg^\ext)\ge \aa_\star(\cB_R^\ext)$ under the constraint $|\p\Omg| = |\p\cB_R|$ was obtained in~\cite{BGGLP26} and
stated in an equivalent form of the Weinstock inequality for the Steklov problem on $\Omg^\ext$.

\subsection*{Structure of the paper}
In Section~\ref{sec:identities}, we provide useful identities for the second eigenfunction on the exterior of the disk. In Section~\ref{sec:geomineq}, we prove a geometric inequality, which is used in the proof of our main result.  Parallel coordinates and trial functions are introduced in Section~\ref{sec:trial}. The central point is the choice of a suitable phase function in the construction of the second trial function.
Estimates of the sesquilinear form evaluated on these trial functions are given in Section~\ref{sec:estimates}. Finally, the main result is proved in Section~\ref{sec:proof}.
 The proof is accompanied by the discussion of related results and the necessity of the assumptions in the main theorem.

\section{Useful identities for the radial profile of the second eigenfunction on the exterior of the disk}\label{sec:identities}
%

The spectral analysis for the Robin Laplacian in the exterior of the disk is performed in~\cite[Section 2.2]{KL24}. Based on this analysis (see~\cite[Proposition 2.4]{KL24}) we have
\begin{equation}\label{eq:disk-threshold}
 \lambda_2^\alpha(\cB_R^{\ext})<0
 \quad\Longleftrightarrow\quad \alpha<-\frac1R.
\end{equation}
In this range $\aa < -\frac1R$, put
\begin{equation}\label{eq:g}
 \lambda:=\lambda_2^\alpha(\cB_R^{\ext})=-k^2,
 \qquad g(t):=K_1(k(R+t)),\qquad t\ge0,
\end{equation}
where $k > 0$ is the unique solution of the transcendental equation
\begin{equation}\label{eq:bessel-boundary}
	kK_1'(kR)=\alpha K_1(kR).
\end{equation}
We remark that the eigenspace of the Robin Laplacian on $\cB^\ext_R$ with $\aa < -\frac1R$ corresponding to the eigenvalue $\lm_2^\aa(\cB_R^\ext)$ is given by the span of the functions $K_1(k(R+t))e^{\pm\ii\tt}$, where $t$ is the distance to the boundary of the disk and $\tt$ is the angular variable in polar coordinates with the origin at the center of the disk.
In particular, $g$ satisfies
\begin{equation}\label{eq:ode}
 -g''-\frac{g'}{R+t}+\frac{g}{(R+t)^2}=\lambda g,
 \qquad g'(0)=\alpha g(0),
\end{equation}
where we used the expression for the Laplacian in polar coordinates and that $K_1(k(R+t))e^{\pm\ii\tt}$
is an eigenfunction  of the Robin Laplacian on $\cB_R^\ext$.
The positivity, monotonicity and exponential decay of $K_1$ give
\begin{equation}\label{eq:g-properties}
 g(t)>0,\qquad g'(t)<0,\qquad
 g(t),\,g'(t)=O\bigl(e^{-kt}(1+t)^{-1/2}\bigr)
 \quad(t\to\infty).
\end{equation}
For example, the identity $K_1'(z)=-K_0(z)-K_1(z)/z<0$ for $z>0$ implies that $g'(t) < 0$.
The Bessel identities and asymptotics used here can be found in
\cite[Sections~10.29 and~10.40]{DLMF}.

\begin{lemma}\label{lem:radial}
The function in \eqref{eq:g} satisfies
\begin{align}
 &\int_0^\infty
  \left((g')^2+\frac{g^2}{(R+t)^2}-\lambda g^2\right)(R+t)\dd t
  +\alpha Rg(0)^2=0,\label{eq:radial-weighted}\\
 &\int_0^\infty
  \left((g')^2+\frac{g^2}{(R+t)^2}-\lambda g^2\right)\dd t
  +\alpha g(0)^2
  =\int_0^\infty\frac{gg'}{R+t}\dd t<0.
  \label{eq:radial-unweighted}
\end{align}
\end{lemma}
\begin{proof}
Multiplying \eqref{eq:ode} by $(R+t)g$ and integrating, the two derivative
terms combine as
\[
 \int_0^\infty\bigl(-(R+t)g''-g'\bigr)g\dd t
 =\int_0^\infty(R+t)(g')^2\dd t+Rg(0)g'(0).
\]
This proves \eqref{eq:radial-weighted}. Without the weight, integration by
parts gives
\[
 \int_0^\infty(-g'')g\dd t
 =\int_0^\infty (g')^2\dd t+g(0)g'(0)= \int_0^\infty (g')^2\dd t+\aa g(0)^2.
\]
Using \eqref{eq:ode} we get
\[
\begin{aligned}
&0=\int_0^\infty \left( -g''-\frac{g'}{R+t}+\frac{g}{(R+t)^2}-\lambda g\right)g\dd t\\
&\qquad = 
\int_0^\infty \left( (g')^2-\frac{g'g}{R+t}+\frac{g^2}{(R+t)^2}-\lambda g^2\right)\dd t+\aa g(0)^2
\end{aligned}
\]
from which the identity 
\eqref{eq:radial-unweighted} follows. We remark that  every boundary term at infinity vanishes by
\eqref{eq:g-properties}. The final sign of~\eqref{eq:radial-unweighted} follows from $g>0$ and $g'<0$.
\end{proof}
\section{An angular variance controlled by the area deficit}
\label{sec:geomineq}
 Let us introduce for a bounded  convex $C^\infty$-smooth domain $\Omg\subset\dR^2$  the parameters $L$, $p$ and $R$ such that
\begin{equation}\label{eq:pR}
 L=|\partial\Omega|=2\pi p,\qquad |\Omega|=\pi R^2.
\end{equation}
Let $\gamma\colon[0,L]\to\partial\Omega$ be a counterclockwise unit-speed
parametrization of the boundary of $\Omg$. For the sake of convenience we extend the function $\gamma$ periodically to the whole real line. The outward normal and tangent for $\p\Omg$, represented upon identification of $\dR^2$ and $\dC$, are $\nu=e^{\ii\theta}$ and $\tau=\ii e^{\ii\theta}$ for a certain smooth function $\tt \colon \dR\arr \dR$ of the arc-length parameter with the normalization $\tt(0) = 0$. This normalization can be achieved by a suitable rotation of the domain $\Omg$.  Here we implicitly assume that $\nu$ and $\tau$ are also extended periodically.
With the convention of non-negative curvature $\kp$ of the boundary for a convex
domain $\Omg$, we get
\begin{equation}\label{eq:geometry}
 \theta'=\kappa\ge0,\qquad
 \theta(L)=2\pi,\qquad
 \nu'=\kappa\tau,\qquad \gamma'=\tau,
\end{equation}
where we extend the curvature also periodically to the whole $\mathbb{R}$.
Choose the constant $c\in\dR$ so that
\begin{equation}\label{eq:psi}
 \psi(s):=\theta(s)-\frac{s}{p}+c,
 \qquad \int_0^L\psi(s)\dd s=0,
 \end{equation}
 and  set
 \[
 C:=\frac p{2\pi}\int_0^L\psi(s)^2\dd s.
 \]
The function $\psi$ is periodic, and
\begin{equation}\label{eq:psi-derivative}
 \psi'=\kappa - \frac{1}{p},
\end{equation}
\begin{lemma}\label{lem:variance}
For every bounded convex $C^\infty$-smooth planar domain,
\begin{equation}\label{eq:variance}
C\le p^2-R^2.
\end{equation}
\end{lemma}
\begin{proof}
First assume $\kappa>0$. We use the normal angle $\vartheta\in[0,2\pi)$ as a boundary
coordinate, that is, the normal vector to the boundary of $\Omg$ at the point with the normal angle coordinate $\vartheta$ is given by $\nu = e^{\ii\vartheta}$ and thus this new coordinate is connected via the arc-length coordinate by the identity $\vartheta = \theta(s)$. For strictly convex domains the normal angle coordinate uniquely determines the point on the boundary. We will denote by $\wh\psi\colon[0,2\pi]\arr\dR$ the function $\psi$ viewed as function of the normal angle parameter and defined through the identity $\wh\psi(\tt(s)) =\psi(s)$ for all $s\in[0,L]$.
Let us also introduce the function
\[
\rho(\vartheta) :=\frac{\dd s}{\dd\vartheta},
\]
that is the derivative of the arc-length coordinate with respect to normal angle coordinate.
Closure and smoothness of the curve $\p\Omg$
give
\[
 \int_0^{2\pi}\rho(\vartheta)e^{\ii\vartheta}\dd\vartheta = \int_0^{L}e^{\ii\tt(s)}\dd s = 0.
\]
Since $\int_0^{2\pi}\rho(\vartheta)\dd\vartheta=L$, the Fourier expansion of $\rho$ has the form
\begin{equation}\label{eq:rho-fourier}
 \rho(\vartheta)=p+\sum_{n\ge2}
       \bigl(a_n\cos n\vartheta+b_n\sin n\vartheta\bigr),
\end{equation}
with real coefficients $(a_n)_{n\ge 2}$ and $(b_n)_{n\ge 2}$.
The first harmonics vanish by the preceding closure identity.

Differentiating $\wh\psi$ with respect to $\vartheta$ we get
\begin{equation}\label{eq:psi-theta}
 \frac{\dd\wh\psi}{\dd\vartheta}=1-\frac{\rho}{p}.
\end{equation}
The arclength centering in \eqref{eq:psi} is also angular centering, since
\[
 0=\int_0^{2\pi}\wh\psi(\vartheta)\rho(\vartheta)\dd\vartheta
   =p\int_0^{2\pi}\wh\psi(\vartheta)\dd\vartheta
       -\frac p2\int_0^{2\pi}(\wh\psi^2)'(\vartheta)\dd\vartheta
   =p\int_0^{2\pi}\wh\psi(\vartheta)\dd\vartheta,
\]
where we used that $\wh\psi(0) = \wh\psi(2\pi)$.
Similarly,
\begin{equation}\label{eq:variance-measures}
 \int_0^L\psi^2(s)\dd s
 =p\int_0^{2\pi}\wh\psi^2(\vartheta)\dd\vartheta
       -\frac p3\int_0^{2\pi}(\wh\psi^3(\vartheta))'\dd\vartheta
 =p\int_0^{2\pi}\wh\psi^2(\vartheta)\dd\vartheta.
\end{equation}
Integrating \eqref{eq:rho-fourier} in \eqref{eq:psi-theta}, with zero
constant term, yields
\begin{equation}\label{eq:psi-fourier}
 \wh\psi(\vartheta)=\frac1p\sum_{n\ge2}
       \left(-\frac{a_n}{n}\sin n\vartheta
                    +\frac{b_n}{n}\cos n\vartheta\right).
\end{equation}
Parseval's identity and \eqref{eq:variance-measures} therefore imply
\begin{equation}\label{eq:C-fourier}
 C=\frac{p^2}{2\pi}\int_0^{2\pi}\wh\psi^2(\vartheta)\dd\vartheta
   =\frac12\sum_{n\ge2}\frac{a_n^2+b_n^2}{n^2}.
\end{equation}
To compute the area, let 
\[
h(\vartheta) := \sup_{x\in\Omega}\left\langle \Big(\begin{smallmatrix}\cos\vartheta\\ \sin\vartheta\end{smallmatrix}\Big),x\right\rangle_{\dR^2}
\]
 be the support function. Then according to~\cite{A98} we have
$h+h''=\rho$ and
\begin{equation}\label{eq:support-area}
 |\Omega|=\frac12\int_0^{2\pi}h(\vartheta)\rho(\vartheta)\dd\vartheta.
\end{equation}
It follows from $h+h'' = \rho$ that the constant Fourier coefficient of $h$ is $p$ and its coefficients for $n\ge2$ are
$a_n/(1-n^2)$ and $b_n/(1-n^2)$. Therefore
\begin{equation}\label{eq:deficit-fourier}
 R^2=p^2-\frac12\sum_{n\ge2}\frac{a_n^2+b_n^2}{n^2-1}.
\end{equation}
Comparison of \eqref{eq:C-fourier} and \eqref{eq:deficit-fourier} proves
\eqref{eq:variance} when $\kappa>0$. 

We now remove the positive-curvature restriction via approximating a general convex domain by strictly convex domains. Without loss of generality we can assume that $0\in\Omg$.
If $r\colon\dS^1\arr\dR_+$ parametrizes the boundary of $\Omega$ in polar coordinates as $(r(\varphi)\cos\varphi, r(\varphi)\sin\varphi)$ with $\varphi\in[0,2\pi)$, then
we consider the family of domains $\Omg_\eps$ with $\eps> 0$ parameterized in polar coordinates by the mapping
\[
 r_\varepsilon(\varphi)
 =\bigl(r(\varphi)^{-2}+\varepsilon\bigr)^{-1/2},\qquad \eps > 0.
\]
It is not difficult to show that the domains $\Omg_\eps$ are strictly convex.
Indeed,
set
\[
q(\varphi):=1+\varepsilon r(\varphi)^2,
\qquad
r_\varepsilon(\varphi)
=\frac{r(\varphi)}{\sqrt{q(\varphi)}}.
\]
By convexity of $\Omega$,
\[
A(\varphi):=r^2+2(r')^2-rr''
=\kappa(\varphi)
\bigl(r^2+(r')^2\bigr)^{3/2}
\ge 0.
\]
Differentiation of $r_\eps$ and application of the polar
curvature formula show that the curvature of $\p\Omg_\eps$ is given by
\[
	\kappa_\varepsilon(\varphi)
	=
	\frac{r_\varepsilon^2
		+2(r_\varepsilon')^2
		-r_\varepsilon r_\varepsilon''}
	{\bigl(r_\varepsilon^2
		+(r_\varepsilon')^2\bigr)^{3/2}}
	=
	\frac{
		q^{3/2}\left[
		qA+\varepsilon r^2
		\bigl(r^2q+(r')^2\bigr)
		\right]}
	{\bigl(r^2q^2+(r')^2\bigr)^{3/2}}
	>0.
\]
Indeed, $A\ge0$, $r>0$, $q>0$, and $\varepsilon>0$,
so the numerator and denominator are strictly positive.
Finally, we observe that $r_\eps\arr r$ in $C^\infty(\dS^1)$ as $\eps\arr0^+$ and thus the perimeters, areas and centered angle variances converge:
$p_\varepsilon\to p$, $R_\varepsilon\to R$, and $C_\varepsilon\to C$.
Passing to the limit proves \eqref{eq:variance} for $\kappa\ge0$.
\end{proof}

We record some useful identities for later purposes. 
\begin{lemma}\label{lem:cancellations}
With the preceding notation,
\begin{align}
 &\int_0^L e^{\ii\theta}\dd s=0,
 \qquad \int_0^L\kappa e^{\ii\theta}\dd s=0,
 \qquad \int_0^L\kappa\psi e^{\ii\theta}\dd s=0,
 \label{eq:closure-cancellations}\\
 &\int_0^L\kappa\psi^2\dd s
   =\frac1p\int_0^L\psi^2\dd s
   =\frac{2\pi C}{p^2}.
 \label{eq:weighted-variance}
\end{align}
\end{lemma}
\begin{proof}
The first identity follows from $\int_0^L\gamma'\dd s=0$. The
second identity follows from $(e^{\ii\theta})'=\ii\kappa e^{\ii\theta}$ and
periodicity. For the third, integration by parts gives
\[
 \int_0^L\kappa\psi e^{\ii\theta}\dd s
 =\frac1\ii\int_0^L\psi(e^{\ii\theta})'\dd s
 =\ii\int_0^L(\kappa-1/p)e^{\ii\theta}\dd s=0.
\]
Finally, use \eqref{eq:psi-derivative} to write
\[
 \int_0^L\kappa\psi^2\dd s
 =\frac1p\int_0^L\psi^2\dd s
       +\frac13\int_0^L(\psi^3)'\dd s,
\]
which proves \eqref{eq:weighted-variance}.
\end{proof}

\section{Parallel coordinates and a distance-dependent phase}
\label{sec:trial}
Convexity makes
\begin{equation}\label{eq:parallel-map}
 \Phi(s,t)=\gamma(s)+t\nu(s),
 \qquad (s,t)\in(\R/L\mathbb Z)\times(0,\infty),
\end{equation}
a global coordinate map onto $\Omega^{\ext}$. Indeed, every 
point in $\Omg^\ext$ has a unique nearest point on $\p\Omg$ and the
connecting segment is normal to the boundary. 
By \eqref{eq:geometry},
\[
 \p_s\Phi=(1+t\kappa)\tau,\qquad \p_t\Phi=\nu.
\]
The absolute value of the determinant of the Jacobian of $\Phi$ given by $w(s,t)=1+t\kappa(s)$ is non-vanishing. The nearest-point property makes
$\Phi$ bijective and thus $\Phi$ is a
global diffeomorphism, which justifies that the coordinates $(s,t)$ on $\Omg^\ext$ are well defined.
Following the lines of~\cite{KL24} we have
\begin{align}
 \|u\|^2_{L^2(\Omg^\ext)}&=\int_0^\infty\int_0^L|u(s,t)|^2w(s,t)\dd s\dd t,
 \label{eq:parallel-mass}\\
 \qf_\alpha^{\Omg^\ext}[u,v]
 &=\int_0^\infty\int_0^L
       \left(\p_t u\overline{\p_t v}\,w
              +\frac{\p_su\overline{\p_sv}}w\right)\dd s\dd t
      +\alpha\int_0^Lu(s,0)\overline{v(s,0)}\dd s.
 \label{eq:parallel-form}
\end{align}
 Total-curvature identity gives
\begin{equation}\label{eq:w-integral}
 \int_0^Lw(s,t)\dd s=L+2\pi t=2\pi P_t,
\end{equation}
where $P_t = p+t$.

Define the phase
\begin{equation}\label{eq:phase}
 \chi(s,t)=\frac{s}{p}+\frac{t}{p+t}\psi(s)
          =\theta(s)+c-\frac p{p+t}\psi(s).
\end{equation}
It satisfies
\begin{equation}\label{eq:phase-derivatives}
 \p_s\chi=\frac{w}{P_t},\qquad
 \p_t\chi=\frac{p\psi}{P_t^2},\qquad
 \chi(s+L,t)=\chi(s,t)+2\pi.
\end{equation}
Thus $\chi$ is normalized arclength on the parallel curve, with an
allowed $t$-dependent choice of angular origin. Two trial functions are
\begin{equation}\label{eq:trials}
 u_0(s,t)=g(t),\qquad u_1(s,t)=g(t)e^{\ii\chi(s,t)},
\end{equation}
where $g$ is as in~\eqref{eq:g}.
The periodicity of $u_0$ and $u_1$ with respect to  variable $s$ makes them well defined on the exterior domain.

Changing variables at fixed $t$ yields
\begin{equation}\label{eq:phase-mean}
 M(t):=\int_0^Le^{\ii\chi(s,t)}w(s,t)\dd s
      =P_t\int_{\chi(0,t)}^{\chi(0,t)+2\pi}e^{\ii\zeta}\dd\zeta=0.
\end{equation}
Consequently
\begin{equation}\label{eq:orthogonality}
 \ip{u_1}{u_0}_{L^2(\Omega^{\ext})}=0,
 \qquad \int_{\partial\Omega}u_1\overline{u_0}\dd s=0.
\end{equation}
The second identity uses $\chi(s,0)=s/p$.
Both functions have the same nonzero squared $L^2$-norm
\begin{equation}\label{eq:common-norm}
 N:=\|u_0\|^2_{L^2(\Omg^\ext)}=\|u_1\|^2_{L^2(\Omg^\ext)}
   =2\pi\int_0^\infty g(t)^2P_t\dd t.
\end{equation}

The derivatives of $u_1$ are
\[
 \p_tu_1=(g'+\ii g\p_t\chi)e^{\ii\chi},\qquad
 \p_su_1=\ii g(\p_s\chi) e^{\ii\chi}.
\]
Since $g,g',\p_t\chi$ are real we get
$|g'+\ii g\p_t\chi|^2 = |g'|^2+|g(\p_t\chi)|^2$. The tangential contribution is
\begin{equation}\label{eq:tangent-energy}
 \int_0^L\frac{|\p_s u_1|^2}{w}\dd s
 =\frac{g^2}{P_t^2}\int_0^Lw\dd s
 =\frac{2\pi g^2}{P_t}.
\end{equation}
By \eqref{eq:weighted-variance},
\begin{equation}\label{eq:psi-w}
 \int_0^L\psi^2w\dd s
 =\left(1+\frac tp\right)\int_0^L\psi^2\dd s
 =\frac{2\pi C P_t}{p^2}.
\end{equation}
The radial phase contribution is therefore
\begin{equation}\label{eq:phase-energy}
 \int_0^Lg^2(\p_t\chi)^2w\dd s
 =\frac{p^2g^2}{P_t^4}\int_0^L\psi^2w\dd s
 =\frac{2\pi Cg^2}{P_t^3}.
\end{equation}
Combining these formulas gives
\begin{align}
 \frac{\qf_\alpha^{\Omg^\ext}[u_0]}{2\pi}
 &=\int_0^\infty g'^2P_t\dd t+\alpha p g(0)^2,
 \label{eq:radial-energy}\\
 \frac{\qf_\alpha^{\Omg^\ext}[u_1]}{2\pi}
 &=\int_0^\infty\left(g'^2P_t+\frac{g^2}{P_t}
                              +\frac{Cg^2}{P_t^3}\right)\dd t
                  +\alpha p g(0)^2.
 \label{eq:angular-energy}
\end{align}
All integrals are finite by \eqref{eq:g-properties}, since $\psi$ is
bounded and $w\ge1$. These computations in particular establish
$u_0,u_1\in H^1(\Omega^{\ext})$.

\section{Diagonal and mixed estimates}
\label{sec:estimates}

For the remainder of the proof define the shifted form
\begin{equation}\label{eq:shifted}
 \hf[u,v]=\qf_\alpha^{\Omg^\ext}[u,v]-\lambda(u,v)_{L^2(\Omg^\ext)},
 \end{equation}
 where $\lambda = \lambda_2^{\aa}(\cB^\ext_R) < 0$ as before. 
 The assumption $\aa < -\frac{1}{R}$ is imposed.
Assume $\Omega$ is not a disk, so $p>R$.

\subsection{The diagonal estimates}

The area bound \eqref{eq:variance} implies the pointwise inequality
\begin{equation}\label{eq:potential}
 \frac1{(p+t)^2}+\frac C{(p+t)^4}
 \le\frac1{(R+t)^2},\qquad t\ge0.
\end{equation}
Indeed, writing $A=p+t$ and $r=R+t$, we have
$C\le p^2-R^2\le A^2-r^2$, and
\[
 \frac C{A^4}
 \le\frac{A^2-r^2}{A^4}
 \le\frac{A^2-r^2}{A^2r^2}
 =\frac1{r^2}-\frac1{A^2}.
\]
Use \eqref{eq:potential} in \eqref{eq:angular-energy}, subtract
$\frac{\lambda N}{2\pi}$, and split $p+t=(R+t)+(p-R)$. The weighted identity
\eqref{eq:radial-weighted} cancels the first part, leaving
\begin{align}
 \frac{\frh[u_1]}{2\pi}
 &\le(p-R)\left[
       \int_0^\infty\left(g'^2+\frac{g^2}{(R+t)^2}-\lambda g^2\right)\dd t
                    +\alpha g(0)^2\right]\notag\\
 &=(p-R)\int_0^\infty\frac{gg'}{R+t}\dd t.
 \label{eq:diagonal-bound}
\end{align}
Define
\begin{equation}\label{eq:DH}
 D:=2\pi(p-R)\int_0^\infty\frac{-gg'}{R+t}\dd t>0,
 \qquad
 H:=2\pi\int_0^\infty g^2
                \left(\frac1{p+t}+\frac C{(p+t)^3}\right)\dd t>0.
\end{equation}
Writing $T=\frh[u_1]$, the exact energy formulas and the common norm yield
\begin{equation}\label{eq:diagonals}
 T\le-D<0,\qquad \frh[u_0]=T-H.
\end{equation}

\subsection{A quadratic cancellation estimate}

The gradient cross term need not vanish. Its control is the step that
allows a general convex obstacle.

\begin{lemma}\label{lem:J}
For $t\ge0$, define
\begin{equation}\label{eq:J-definition}
 J(t)=\int_0^L\kappa(s)e^{\ii\chi(s,t)}\dd s.
\end{equation}
Then
\begin{equation}\label{eq:J-bound}
 |J(t)|\le\frac{\pi C}{(p+t)^2}.
\end{equation}
\end{lemma}
\begin{proof}
Let $a_t=p/(p+t)$. By \eqref{eq:phase},
$e^{\ii\chi}=e^{\ii c}e^{\ii\theta}e^{-\ii a_t\psi}$. The cancellations in
\eqref{eq:closure-cancellations} remove the constant and linear terms:
\begin{equation}\label{eq:J-remainder}
 J(t)=e^{\ii c}\int_0^L\kappa e^{\ii\theta}
        \bigl(e^{-\ii a_t\psi}-1+\ii a_t\psi\bigr)\dd s.
\end{equation}
For real $x$, Taylor's integral remainder gives
\[
 e^{-\ii x}-1+\ii x
 =-x^2\int_0^1(1-r)e^{-\ii r x}\dd r,
 \qquad |e^{-\ii x}-1+\ii x|\le x^2/2.
\]
Since $\kappa\ge0$, we can apply this estimate under the integral in
\eqref{eq:J-remainder}. Using \eqref{eq:weighted-variance},
\[
 |J(t)|\le\frac{a_t^2}{2}\int_0^L\kappa\psi^2\dd s
 =\frac{p^2}{2(p+t)^2}\frac{2\pi C}{p^2}
 =\frac{\pi C}{(p+t)^2}.
\]
This argument is valid directly when $\kappa$ vanishes at some points of the boundary, because no inverse
normal-angle parametrization has been used.
\end{proof}

\subsection{The off-diagonal form}

With the convention that the form is linear in its first argument, set
$b=\hf[u_1,u_0]$. Both the $L^2$ and Robin boundary terms vanish by
\eqref{eq:orthogonality}. Since $\p_su_0=0$, only the $t$-derivative
term remains:
\begin{align}
 b
 &=\int_0^\infty\int_0^L
       (g'+\ii g(\p_t\chi))e^{\ii\chi}g'w\dd s\dd t\notag\\
 &=\ii\int_0^\infty gg'
                 \int_0^L(\p_t\chi) e^{\ii\chi}w\dd s\dd t.
 \label{eq:cross-first}
\end{align}
Differentiating $M(t)=0$ in \eqref{eq:phase-mean} gives
\begin{equation}\label{eq:M-derivative}
 0=M'(t)=\ii\int_0^L(\p_t\chi) e^{\ii\chi}w\dd s+J(t).
\end{equation}
Differentiation under this finite $s$ integral is justified by the
regularity of the phase and the boundary. Consequently
\begin{equation}\label{eq:cross-exact}
 b=-\int_0^\infty gg'J(t)\dd t.
\end{equation}
The integral is absolutely convergent. Lemma~\ref{lem:J} and
$-gg'>0$ imply
\begin{equation}\label{eq:cross-bound-first}
 |b|\le\pi C\int_0^\infty\frac{-gg'}{(p+t)^2}\dd t.
\end{equation}
The remaining comparison is elementary. Since $p\ge R$ and $t\ge0$,
\begin{align}
 2(p+t)^2-(p+R)(R+t)
 &=(p-R)(2p+R)+(3p-R)t+2t^2\notag\\
 &\ge0.
 \label{eq:algebra}
\end{align}
Together with $C\le(p-R)(p+R)$, this yields
\begin{equation}\label{eq:C-comparison}
 \frac C{(p+t)^2}\le\frac{2(p-R)}{R+t}.
\end{equation}
Combining \eqref{eq:cross-bound-first} and \eqref{eq:C-comparison} proves
\begin{equation}\label{eq:cross-final}
 |b|\le2\pi(p-R)\int_0^\infty\frac{-gg'}{R+t}\dd t
      =D\le-T.
\end{equation}

\section{Completion of the variational proof}
\label{sec:proof}

\begin{proof}[Proof of Theorem~\ref{thm:main}]
If $\Omega$ is a disk, translation invariance gives equality. Suppose
that $\Omega$ is not a disk. When $\alpha\ge-1/R$,
\eqref{eq:disk-threshold} and the convention for min--max levels give
\[
 \lambda_2^\alpha(\Omega^{\ext})\le0
 =\lambda_2^\alpha(\cB_R^{\ext}).
\]
It remains to treat $\alpha<-1/R$, so all the preceding constructions
apply.

Consider $\mathcal L=\Span_{\C}\{u_0,u_1\}$. The functions are nonzero
and orthogonal by \eqref{eq:orthogonality}--\eqref{eq:common-norm}, so
$\dim_{\C}\mathcal L=2$. For $z_0,z_1\in\C$,
\begin{align}
\frh[z_0u_0+z_1u_1]
 &=(T-H)|z_0|^2+T|z_1|^2
        +2\Rea\bigl(z_1\overline{z_0}b\bigr)\notag\\
 &\le(T-H)|z_0|^2+T|z_1|^2+2|b||z_0||z_1|\notag\\
 &\le T(|z_0|-|z_1|)^2-H|z_0|^2.
 \label{eq:negative-space}
\end{align}
Here the last step is \eqref{eq:cross-final}. Since $T<0$ and $H>0$,
the final expression is strictly negative whenever $(z_0,z_1)\ne(0,0)$:
if $z_0\ne0$, the second term is negative; if $z_0=0$ and $z_1\ne0$,
the first term is negative.

The unit sphere of this finite-dimensional trial space is compact. Thus
strict negativity gives
\[
 \sup_{0\ne u\in\mathcal L}\frac{\qf_\alpha^{\Omg^\ext}[u]}{\|u\|^2_{L^2(\Omg^\ext)}}
 <\lambda.
\]
Applying \eqref{eq:minmax} proves
\[
 \lambda_2^\alpha(\Omega^{\ext})
 <\lambda=\lambda_2^\alpha(\cB_R^{\ext}),
\]
which establishes the assertion and its equality statement.
\end{proof}

\begin{remark}[Where convexity enters]
Convexity is used in three places: the global exterior parallel
coordinates, the area--variance inequality, and the Taylor estimate with
the nonnegative measure $\kappa\dd s$. The proof does not invoke an
orthogonality condition against the unknown ground state. Instead, it
controls the whole shifted form on a two-dimensional trial space.
\end{remark}

\begin{remark}[Relation to the earlier parallel-coordinate method]
The trial functions in \cite[Section~4]{KL24} use complexified boundary tangent as
their angular factor, and their mixed terms vanish exactly. In the
present construction the phase varies with $t$, producing the term
$C/(p+t)^4$ in the effective angular potential. The mixed term is
generally nonzero, but its quadratic estimate \eqref{eq:J-bound} is
small enough for \eqref{eq:negative-space}. The same disk radial profile
$g$ is used in both trial functions.
\end{remark}

\begin{remark}[Scope and threshold comparisons]
Theorem~\ref{thm:main} addresses the fixed-area part of
\cite[Conjecture~1.2]{KL24} within the convex $C^\infty$-smooth domains. Its argument does not establish the
fixed-perimeter comparison,  which is a stronger claim and does not follow by scaling. Exterior Steklov inequalities and their
connection with the Robin threshold are developed in
\cite{BGGLP26}. Those results provide complementary comparisons for 
the critical coupling constant under fixed perimeter constraint. 
\end{remark}

\subsection*{Acknowledgments}
The first author was supported  
by the grant no. 26-21940S of the Czech Science Foundation.

\subsection*{AI usage disclosure}
We disclose the use of ChatGPT 6 Astra. The proof of the main result of the present manuscript was developed after a few prompts asking to generalize the method of our paper~\cite{KL24} in order to prove the conjecture stated in the same paper. We have checked and rewritten all AI-assisted arguments and we take responsibility for their correctness.

\end{document}